\documentclass[12pt]{amsart}
\usepackage{amsmath}
\usepackage{graphicx}
\usepackage{amsfonts}
\usepackage{amstext}
\usepackage{amsopn}
\usepackage{amsxtra}
\usepackage{upref}
\usepackage{amsthm}
\usepackage{amsmath}
\usepackage{amssymb}

\newtheorem{thm}{Theorem}%[section]

\newtheorem{lemma}{Lemma}

\theoremstyle{definition}

\theoremstyle{remark}

\numberwithin{equation}{section}

\newcommand{\va}{\varphi}

\newcommand{\noi}{\noindent}
\newcommand{\sm}{\smallskip}
\newcommand{\me}{\medskip}
\newcommand{\bi}{\bigskip}

\newcommand{\bz}{\bar{z}}

\def\D{{\mathbb D}}

\newcommand{\vphi}{\varphi}

\begin{document}

\title{A note on convex conformal mappings}
\author{Martin Chuaqui and Brad Osgood}
\thanks{The first
author was partially supported by Fondecyt Grant  \#1150115.
\endgraf  {\sl Key words: Convex mapping, Poincar\'e metric, level set, curvature.}
\endgraf {\sl 2000 AMS Subject Classification}. Primary: 30C45;\,
Secondary: 30C80, 30C62.}
%
%\address{Facultad de Matem\'aticas\\ Pontificia Universidad Cat\'olica de Chile\\
%Casilla 306, Santiago 22, CHILE.} \email{mchuaqui@mat.puc.cl}

%\thanks{}%
%\subjclass{}%
%\keywords{}%

%\date{}%
%\dedicatory{}%
%\commby{}%
% ----------------------------------------------------------------
\begin{abstract}
We establish a new characterization for a conformal mapping of the unit disk $\D$ to be convex, and identify the mappings onto a half-plane or a parallel strip as extremals. We also show that, with these exceptions, the level sets of $\lambda$ of the Poincar\'e metric $\lambda|dw|$ of a convex domain are strictly convex.

\end{abstract}
\maketitle

The purpose of this short article is to present a new sharp characterization of conformal mappings of the unit disk $\D$ onto  convex domains with some implications for the Poincar\'e metric of the image. In particular, we will improve on an inequality obtained in \cite{cdo:convex}, where the classical characterization of convexity
\begin{equation} \label{eq:classical-convexity-condition}
\mbox{Re}\left\{1+z\frac{f''}{f'}(z)\right\} \ge 0 \tag{1}
\end{equation}
was shown to imply the stronger inequality
\begin{equation} \label{eq:stronger-convexity-condition}
\mbox{Re}\left\{1+z\frac{f''}{f'}(z)\right\}\geq \frac14(1-|z|^2)\left|\frac{f''}{f'}(z)\right|^2 \, . \tag{2}
\end{equation}
Let $Sf$ be the Schwarzian derivative of $f$. Our first result is that, in fact,

\begin{thm} The function $f$ is a convex mapping of $\D$ if and only if
\begin{equation} \label{eq:new-convexity-condition}
\mbox{Re}\left\{1+z\frac{f''}{f'}(z)\right\}\geq \frac14(1-|z|^2)\left(2|S\!f(z)|+\left|\frac{f''}{f'}(z)\right|^2\right).  \tag{3}
\end{equation}
 If equality holds at a single point in \eqref{eq:new-convexity-condition} then it holds everywhere and $f$ is a mapping either onto a half-plane or a parallel strip.
 \end{thm}
\noindent   Note that \eqref{eq:new-convexity-condition}  reduces to \eqref{eq:stronger-convexity-condition} when $f$ is a M\"obius transformation.

A second result is an application of  \eqref{eq:new-convexity-condition} to a property of the Poincar\'e metric for convex regions. Recall that the Poincar\'e metric on $f(\D)$ is defined by $\lambda(w)|dw| = |dz|/(1-|z|^2)$, $w=f(z)$.
We will show that, except for a  half-plane or a parallel strip, the level sets of  $\lambda$  have  strictly positive curvature, or equivalently that the sets in $\D$ defined by
$
(1-|z|^2)|f'(z)|=\mbox{constant}
$
have this property relative to the conformal metric $|f'||dz|$. The presence of the Schwarzian term in (3) is crucial for establishing this.

\begin{thm} The level sets of $\lambda(w)$ in a convex domain have nonnegative curvature. If the curvature of any level set is zero at a single point then the domain is a parallel strip or a half-plane and all level sets have zero curvature.
\end{thm}

This refines the results in \cite{cf:convex} and  \cite{mk:convex}, where it is shown that on convex regions the function $1/\lambda$ is concave, or equivalently that $\log\lambda$ is convex. It follows from these earlier results that the sets $\lambda\leq c$ are convex, but it does not rule out flat parts of the curve $\lambda=c$ or isolated  points where the curvature vanishes. %Our result is that the curvature is strictly positive everywhere unless the domain is  a half-plane or a parallel strip.

\me

\begin{proof}[Proof of Theorem 1] The sufficiency follows at once as \eqref{eq:new-convexity-condition} is stronger than \eqref{eq:classical-convexity-condition}. Suppose next that $f$ is convex.
Via \eqref{eq:classical-convexity-condition}  we know that
\[
1+z \frac{f''}{f'}(z)=\frac{1+h(z)}{1-h(z)}
\]
for some holomorphic $h\colon\D\rightarrow \D$ with $h(0)=0$. We appeal to Schwarz's lemma. The function $\vphi(z)=h(z)/z$ is holomorphic, maps $\D$ into $\overline{\D}$, and
 \[
 \va(z) = \frac{f''(z)/f'(z)}{2+zf''(z)/f'(z)}\,. % \quad  P(z) = \bar{z}-\frac12(1-|z|^2)T\!f(z)\,.
\]
One possibility is $|\vphi| \equiv 1$. In this case $f$ is a half-plane mapping, $Sf=0$, and \eqref{eq:new-convexity-condition}, really \eqref{eq:stronger-convexity-condition}, holds with equality for all $z$.

If $|\vphi|\not\equiv 1$ then
\[
 \frac{|\va'(z)|}{1-|\va(z)|^2}\leq \frac{1}{1-|z|^2}. \tag{4}
 \]
This implies
\[ \label{eq:kim-minda}
% (1-|z|^2)^2|Sf(z)| +2\left| \bar{z}-\frac{1}{2}(1-|z|^2)\frac{f''}{f'}(z)\right|^2 \le 2 \tag{4}
(1-|z|^2)^2|Sf(z)| +2\left| p(z) \right|^2 \le 2 \tag{5}
\]
after a short calculation, where we have written
\[
p(z) = \bar{z}-\frac{1}{2}(1-|z|^2)\frac{f''}{f'}(z).
\]
In turn, on expanding $|p(z)|^2$, (5) can be rearranged to yield (3) (and vice versa). The two inequalities are equivalent, but the important point for our work is that in (3) the factor $1-|z|^2$ occurs to the first power, not to the second.

Suppose now that equality holds in (3) at one point, and suppose also that $|\vphi|<1$ (the case $|\vphi|\equiv 1$ having been analyzed). Equality in (3) at a point implies equality in (5) at a point, and then also equality  in (4) at a point. Thus $\vphi$ is a M\"obius transformation of $\D$ to itself and equality holds everywhere in (3), (4) and (5). Furthermore, it follows from Lemma 1 in \cite{cdo:convex} that $f$ maps $\D$ onto a parallel strip.

\end{proof}
The proof shows that (5) is also a necessary and sufficient condition for a mapping to be convex. This was originally established in \cite{mk:convex} and also proved, essentially as above, in \cite{cdo:convex}. Actually, the loop of implications is (1) $\implies$ (4) (or $|\vphi|\equiv 1$) $\implies$  (5) $\implies$ (3) $\implies$ (1), and also (1) $\iff$ (2), so all are equivalent  to $f$ being a convex mapping.

\me

We now turn to the convexity property of the Poincar\'e metric.
The level set
$\lambda(w)=1/c $
corresponds under $f$ to the curve in $\D$ where
\[
{(1-|z|^2)|f'(z)|}=c \, .
\]
This will be a smooth curve provided $\nabla ((1-|z|^2)|f'(z)|) \neq 0$ there, and this is equivalent to the condition
$p \ne 0$. Thought of as a vector, the complex number $p$ is normal to the curve.

For the proof of Theorem 2 we need a formula for curvature that in itself is not particular to convexity. %\bar{z}-\frac12(1-|z|^2)p\neq 0 \, .

\begin{lemma} Let $f$ be locally injective in $\D$, and let $\gamma\subset\D$ be the level set
\[
(1-|z|^2)|f'(z)|=c \, ,
\]
for a constant $c$. Suppose that $p\neq 0$ on $\gamma$. Then
\begin{equation*}
|p(z)|k(z)=1+\frac14(1-|z|^2)|\frac{f''}{f'}(z)|^2+\frac{1-|z|^2}{2|p(z)|^2}{\rm Re}\{\overline{p(z)}^2 S\!f(z)\} \, ,\tag{6}
\end{equation*}
where $k$ is the euclidean curvature of $\gamma$.
\end{lemma}

\begin{proof}

Because $p\neq 0$ on $\gamma$, we may choose a Euclidean arclength parametrization $z=z(s)$, oriented so that the normal direction $p$ points to the right of $z'(s)$. Let $q = \overline{p}$ and $\hat{q} = q/|q|$. With the given orientation of $\gamma$ we have that
\begin{equation*} \label{eq:Psi-relationships}
z'=-i\hat{q} \quad  \text{and} \quad
z''=-k\hat{q} \, , \tag{7}
\end{equation*}
with $k\geq 0$ if and only if the level set is convex.

Differentiating $(1-|z(s)|^2)|f'(z(s))|=c$ once we obtain
\[
 \mbox{Re} \left\{z'\,\frac{f''}{f'}(z)\right\}=2\frac{\mbox{Re}\{\bz z'\}}{1-|z|^2}\, ,
 \]
while a second differentiation yields
\[
\begin{aligned}
\mbox{Re}\left\{(z')^2\left(\frac{f''}{f'}\right)'(z)\right\}+\mbox{Re}\left\{z''\frac{f''}{f'}(z)\right\}&=2\frac{1+\mbox{Re}\left\{\bz z''\right\}}{1-|z|^2}+
4\left(\frac{\mbox{Re}\{\bz z'\}}{1-|z|^2}\right)^2\\
&= 2\frac{1+\mbox{Re}\{\bz z''\}}{1-|z|^2}+\mbox{Re}\left\{z'\,\frac{f''}{f'}(z)\right\}^2\, .
\end{aligned}
\]
Rewrite the last term on the right hand side as
\[
\mbox{Re}\left\{\left(z'\,\frac{f''}{f'}(z)\right)^2\right\}+\mbox{Im}\left\{z'\,\frac{f''}{f'}(z)\right\}^2\,
\]
to get
\[
\mbox{Re}\left\{(z')^2S\!f(z)\right\}+\mbox{Re}\left\{z''\,\frac{f''}{f'}(z)\right\}=2\frac{1+\mbox{Re}\{\bz z''\}}{1-|z|^2}+\frac12\mbox{Re}\left\{\left(z'\,\frac{f''}{f'}(z)\right)^2\right\}
+\mbox{Im}\left\{z'\,\frac{f''}{f'}(z)\right\}^2 \, .
\]
Since
\[
\frac12\mbox{Re}\left\{\left(z'\,\frac{f''}{f'}(z)\right)^2\right\}=\frac12\mbox{Re}\left\{z'\,\frac{f''}{f'}(z)\right\}^2-\frac12 \mbox{Im}\left\{z'\,\frac{f''}{f'}(z)\right\}^2\,,
\]
we obtain
\[
\mbox{Re}\left\{(z')^2S\!f(z)\right\}+\mbox{Re}\left\{z''\,\frac{f''}{f'}(z)\right\}=2\frac{1+\mbox{Re}\{\bz z''\}}{1-|z|^2}+\frac12\left|\frac{f''}{f'}(z)\right|^2 \, ,\]
which we further rewrite as
\[
-2\mbox{Re}\left\{z''p(z)\right\}=\frac{2}{1-|z|^2}+\frac12\left|\frac{f''}{f'}(z)\right|^2-\mbox{Re}\left\{(z')^2S\!f(z)\right\}\, .
\]
Using (7), this is the equation in the lemma.
\end{proof}

The issue in establishing strict convexity is the presence of critical points for $\lambda$. These correspond to points in $\D$ where $p(z)=0$. Now, convex mappings satisfy
\[
\sup_{|z|<1}(1-|z|^2)^2|S\!f(z)|\leq 2\,, \tag{8}
\]
see \cite{nehari:convex} and \cite{cdo:convex}.  The results in \cite{co:aw} thus apply, namely that $\lambda$ has at most one critical point, with the exception of a parallel strip where $\nabla \lambda =0$ all along the central line.  For unbounded convex domains, more generally for unbounded domains coming from (8), there are \emph{no} critical points, except again for a parallel strip.  When it exists,  the unique critical point corresponds to the absolute minimum of $\lambda$.

\begin{proof}[Proof of Theorem 2]
We analyze level sets away from the unique critical point of $\lambda$, if there is one.
Let $\kappa$ be the curvature of $\gamma$ relative to the metric $|f'||dz|$ (which is the Euclidean curvature of $f(\gamma)$) and let $\sigma=\log|f'|$. Then
\[
 e^{\sigma}\kappa=k-\frac{\partial}{\partial n}\sigma \, ,
 \]
  where $\partial/\partial n$ is the derivative of $\sigma$ in the direction of the
normal  to $\gamma$, i.e., in the direction $-\hat{q}$. Hence
\[
\begin{aligned}
e^{\sigma(z)}\kappa(z)&=k(z)+2\mbox{Re}\{\hat{q}(z)\partial_z\sigma(z)\}\\
&=k(z)+\frac{\mbox{Re}\{z\,\frac{f''}{f'}(z)\}}{|p(z)|}-\frac{1}{2|p(z)|}(1-|z|^2)\left|\frac{f''}{f'}(z)\right|^2\\
&=\frac{1}{|p(z)|}\left(k|p(z)|+\mbox{Re}\left\{z\,\frac{f''}{f'}(z)\right\}-\frac12(1-|z|^2)\left|\frac{f''}{f'}(z)\right|^2\right) \, .
\end{aligned}
\]
We replace the expression for $|p|k$ from the lemma, and obtain
$$
e^{\sigma(z)}\kappa(z)=\frac{1}{|p(z)|}\left(1-\frac14(1-|z|^2)\left|\frac{f''}{f'}(z)\right|^2+\mbox{Re}\left\{z\frac{f''}{f'}(z)\right\}-\frac12(1-|z|^2)\mbox{Re}\{(z')^2Sf(z)\}\right)$$
$$
=\frac{1}{|p(z)|}\left(\mbox{Re}\left\{1+z\frac{f''}{f'}(z)\right\}-\frac14(1-|z|^2)\left|\frac{f''}{f'}(z)\right|^2-\frac12(1-|z|^2)\mbox{Re}\{(z')^2Sf(z)\}\right)$$
$$\geq \frac{1}{|p(z)|}\left(\mbox{Re}\left\{1+z\frac{f''}{f'}(z)\right\}-\frac14(1-|z|^2)\left|\frac{f''}{f'}(z)\right|^2-\frac12(1-|z|^2)|Sf(z)|\right) \geq 0 \, ,$$
the final inequality holding precisely because of Theorem 1.

We claim that if $\kappa=0$ at one point, then $f$ maps $\D$ onto a half-plane or onto a parallel strip. Indeed, if the curvature vanishes at some point, then all inequalities used to derive that $\kappa\geq 0$ must be equalities. Referring to the proof of Theorem 1, this implies that the function $\va$ must be  a constant of absolute value $1$ or an automorphism of the disk. In the first case, $f(\D)$ is a half-plane, where the level sets of $\lambda$ are all straight lines parallel to the boundary.
In the second case $f(\D)$ is a parallel strip. The Poincar\'e metric on the model strip $|\mbox{Im}\, y|<\pi/2$, is  $\lambda|dw|=\sec y|dw|$, $w=x+iy$. The axis of symmetry $y=0$ is where  $\lambda$  has its absolute minimum, and $\nabla\lambda=0$ there. Any other level set will consist of a pair of horizontal lines $y=\pm a$, for some $a\in(0,\pi/2)$.   In summary, if the curvature is zero at one point of a level set then it is zero at all points of all level sets.

\end{proof}

We are happy to thank Peter Duren for his interest in this work.

\me
\bibliographystyle{plain}
%\bibliography{refs}

\bi
\noi
{\small Facultad de Matem\'aticas, Pontificia Universidad Cat\'olica de Chile,
 \email{mchuaqui@mat.uc.cl}}

\sm
\noi
{\small Department of Electrical Engineering, Stanford University,
 \email{osgood@stanford.edu}}

\end{document}